\documentclass[12pt]{article}
\usepackage{amsthm}
\usepackage{amsmath}
\usepackage{amssymb} 
\usepackage{comment}
\usepackage{fullpage}
\usepackage{hyperref}
\usepackage[linesnumbered,ruled,vlined]{algorithm2e}
\newtheorem{thm}{Theorem}[section]
\newtheorem{lem}[thm]{Lemma}
\theoremstyle{remark}

\newcommand{\Fn}{\mathbb{F}_{q^n}}
\newcommand{\F}{\mathbb{F}_q}
\newcommand{\x}{\xi^{-1}}
\newcommand {\w}{\omega}

\DeclareMathOperator{\rad}{rad}

\title{Primitive elements with prescribed traces}
\author{Andrew R. Booker\\
{\small University of Bristol, England}\\
{\small andrew.booker@bristol.ac.uk}
\and
Stephen D. Cohen\footnote {Postal address: 6 Bracken Road, Portlethen, Aberdeen AB12 4TA, Scotland.}\\
{\small University of Glasgow, Scotland}\\
  {\small Stephen.Cohen@glasgow.ac.uk}
\and
Nicol Leong\\ {\small The University of New South Wales Canberra,}\\{\small Australia }\\
{\small  nicol.leong@adfa.edu.au }
\and
  Tim Trudgian\footnote{Supported by Australian Research Council Future Fellowship FT160100094.}\\ {\small The University of New South Wales Canberra,}\\{\small Australia }\\
  {\small  t.trudgian@adfa.edu.au } }

\begin{document}
\maketitle
\begin{abstract}
Given a prime power $q$ and a positive integer $n$, let  $\Fn$  denote the finite field with $q^n$ elements.  Also  let $a,b$ be arbitrary members of the ground field $\F$.  We investigate the existence of a non-zero element $\xi \in \Fn$ such that $\xi+ \xi^{-1}$ is primitive and $T(\xi)=a$, $T(\xi^{-1})=b$, where $T(\xi)$ denotes the trace of $\xi$ in $\F$.  This was  a question intended to be addressed by  Cao and Wang in 2014.  Their work  dealt instead with another problem already in the literature.  Our solution deals with all values of $n \geq 5$. 

A related study involves the cubic extension $\mathbb{F}_{q^{3}}$
of $\mathbb{F}_{q}$. We show that if $q\geq 8\cdot 10^{12}$ then, for
any $a\in \mathbb{F}_{q}$, we can find a primitive element $\xi \in
\mathbb{F}_{q^{3}}$  such that $\xi + \xi^{-1}$ is also a primitive
element of $\mathbb{F}_{q^{3}}$, and for which the trace of $\xi$ is
equal to $a$. The improves a result of Cohen and Gupta. Along the way
we prove a hybridised lower bound on prime divisors in various residue
classes, which may be of interest to related existence questions.
\end{abstract}

\section{Introduction}
A {\em primitive element} of a finite field is a generator of the (cyclic) multiplicative group of the field. Let $\Fn$  denote the finite field with $q^n$ elements, where (as throughout)  $q$ is a prime power and $n$ is a positive integer.  Also let $a,b$ be arbitrary members of the ground field $\F$. The second author \cite{Co} investigated the problem of the existence of a primitive element $\xi \in \Fn$  such that $T(\xi)=a$ and $T(\xi^{-1})=b$ and gave a complete answer when $n \geq 5$.  Call this problem P1.    Motivated by a connection to Gauss periods and their cryptographical applications (see, e.g., \cite {GGP}),  Cao and Wang  \cite{CW} intended to treat a variant of P1, the difference being that, rather than $\xi$ being prescribed to be  primitive, it is the element $\xi+\xi^{-1}$ that is required to be primitive  (where, of course, we have $\xi \neq 0$). Call this Problem P2; see also the introduction to \cite{SharmaRS} for more details.   Inadvertently, the work of Cao and Wang actually considers P1 with only partial results (see the proof of Theorem \ref{Nebound} below).  

Given that P2 is evidently more challenging than P1, we aim at providing some correct results on P2.  Specifically we prove the following theorem in which $Q$ denotes the set of pairs $(q,n)$ such that,  for any $a,b \in \F$ there exists an element $\xi \in \Fn$ with $T(\xi)= a, T(\xi^{-1})=b$ and for which $\xi+\x$ is a primitive element of $\Fn$.

\begin{thm}\label{main} 
We have $(q,n)\in Q$ for every $q$ and every $n\ge5$, with the exception of
the pair $(q,n)=(2,6)$.
\end{thm}

The outline of this paper is as follows. In \S \ref{prelim} we outline the necessary theory to translate the problem P2 into something tractable. We develop this then produce a lower bound for the problem P2 in \S \ref{cheek}. In \S \ref{6up} we consider $n\geq 6$, and, in \S \ref{sec5}, we outline a new method to tackle  $n=5$. This method allows one to create better, hybridised bounds on certain prime divisors of $q^{n} -1$ and may be of interest to related problems. We detail some computations in \S 6, which finally  prove Theorem \ref{main}. In \S 7 we conclude with an application to a related problem considered by Cohen and Gupta \cite{Gupta}.

Throughout this paper we let $\omega(n)$ denote the number of distinct primes factors of $n$. We also let $\phi(n)$ denote Euler's function and $\mu(n)$ denote the M\"{o}bius function.

\section{Preliminaries}\label{prelim}
Given a pair $(q,n)$, let $e$ be any divisor of $q^n-1$.  A non-zero
element $\xi$ of $\Fn$ is defined to be {\em $e$-free} if $\xi=
\beta^d$, where $\beta \in \Fn$ and $d\mid e$ implies $d=1$.  Observe
that $\xi$ is $e$-free if and only if $\xi$ is $\rad(e)$-free,
where $\rad(e)$ denotes the product of the distinct primes
dividing $e$.  Then  $\xi$ is primitive if and only if it is
$(q^n-1)$-free.  Further for any $e$, $\theta(e) = \frac{\phi(e)}{e}$.
When $e\mid q^n-1$ then $\theta(e)$ is the proportion of elements of $\Fn$ that are $e$-free.

As in \cite{CW}, set
\[ P(d,\xi)= \frac{\mu(d)}{\phi(d)} \sum_{\psi_d} \psi_d(\xi), \]
where the sum is defined over all multiplicative characters $\psi_d$ of order $d$.
For $e\mid q^n-1$, write

\[P_e(\xi) = \theta(e) \sum_{d\mid e} P(d,\xi).\]
 Thus
 \[ P_e(\xi) = \begin{cases} 1, & \textrm{if $\xi$ is $e$-free},\\
 0, & \textrm{otherwise}.
 \end{cases}\]
When $e=q^n-1$ write $P(\xi)$ for $P_e(\xi)$ to yield  the characteristic function for primitive elements.
For $a,b \in \F$, write
\[L(\xi) = \frac{1}{q^2} \sum_{u,v \in \F}\lambda(-ua-vb) \hat{\lambda}(u\xi+v\xi^{-1}),\]
where $\lambda$ is the canonical additive character of $\F$ and $\hat{\lambda}$ is its lift to $\Fn$.
Thus
\[ L(\xi) = \begin{cases} 1, & \textrm{if $T(\xi)=a$ and $T(\xi^{-1})=b$},\\
0, &  \textrm{otherwise}.
\end{cases}\]

In  Lemma 3.2 of \cite{CW},  the authors give a consequence of Theorem 1.1 of \cite {CP}.  We go back to a version of the latter theorem (proved by Castro and Moreno \cite{CM}) given in (1.4) of \cite{CPre}.
\begin{lem} \label{charest}
Let $f,g$ be rational functions in $\F(x)$ and  $\lambda, \psi$ (respectively)  be non-trivial  additive and multiplicative characters on $\F$. Let $\mathcal{S}$ be the set of poles (in $\F$) of $f$ and $g$.     Let $l$  be the number of distinct zeros and (non-infinite) poles of $g$.  Further, let  $l_1$ be  the number of poles of $f$ (including $\infty$) and $l_0$ be the sum of the multiplicities of these poles.  Finally, let $l_2$ be the number of non-infinite poles of $f$ which are zeros or poles of $g$. Define
\[ S(\lambda, f; \psi,g)=n \sum_{x \in \F\setminus \mathcal{S}} \lambda(g(x))\psi(f(x)).\]
Then 
\[| S(\lambda, f; \psi,g)| \leq ( l_0+l+l_1-l_2-2) q^{\frac{1}{2}}. \]

 \end{lem}

\section{A lower bound in Problem P2}\label{cheek}

Given $a,b \in\F$ we seek a lower  bound for the number of $\xi \in
\Fn^*$ such that $\xi+\xi^{-1}$ is primitive with $T(\xi)=a$ and
$T(\xi^{-1})=b$ (so that $T(\xi+\xi^{-1}) =a+b$).  It is convenient to
consider a more general quantity. So, for any $e\mid q^n-1$ define $N_e=N_e(a,b)$ as the number of $\xi \in \Fn$ such that $\xi+\xi^{-1}$ is $e$-free with $T(\xi)=a$ and $T(\xi^{-1})=b$.   Ultimately, we want to discover when $N=N_{q^n-1}$ is positive: this number was denoted by $N_U$ in  \cite{CW}.  Define $Q$ to be the set of pairs $(q,n)$ with $q$ a prime power and $n$ a positive integer such that, for any $a,b \in \F$, there exists $\xi \in \Fn^{*}$ such that $T(\xi)=a, T(\xi^{-1}) =b$ and $\xi+\x$ is primitive.  

\begin{thm} \label{Nebound}
Let $q$ be a prime power  and $n$ be a positive integer.   Suppose that $a,b$ are arbitrary  elements of $\F$ and $e$ a divisor of $q^n-1$.  Then the number of $\xi$ in $\Fn^*$ such that $T(\xi)=a$ and $T(\xi^{-1})=b$ with $\xi+\xi^{-1}$  $e$-free, satisfies 
\[ N_e \geq  \theta(e)q^{\frac{n}{2}} \left( q^{\frac{n}{2}-2} -2^{\omega(e)+2}   \right).\]

In particular, we have $(q,n) \in Q$ provided that 
\begin{equation} \label{Nbound}
  q^{\frac{n}{2}-2} >2^{\omega(q^n-1)+2}.
    \end{equation}

\end{thm}
\begin{proof}
In the notation of \S \ref{prelim} we have 
\begin{equation}\label{Ne}
N_e= \sum_{\xi \in \Fn^*} P(\xi+\xi^{-1})L(\xi),
\end{equation}
which, when $e=q^n-1$ is the same as the expression in line -5 in \cite[p.\ 342]{CW}.  Unfortunately, the next line of \cite{CW}  equates (\ref{Ne}) (with $e=q^n-1$)  to
\[ \frac{\theta(q^n-1)}{q^2}\sum_{\xi \in \Fn^*}\sum_{d\mid q^n-1} P(d,\xi)\sum_{u,v \in \F}\lambda(-ua-vb)\hat{\lambda}(u\xi+v\xi^{-1}),\]
but this evidently relates to the relevant  expression for Problem P1 and not P2.    Returning to a general divisor $e$ of $q^n-1$ we see that, in fact,
\[N_e=\frac{\theta(e)}{q^2}\sum_{\xi \in \Fn^*}\sum_{d\mid e} P(d,\xi+\xi^{-1})\sum_{u,v \in \F}\lambda(-ua-vb)\hat{\lambda}(u\xi+v\xi^{-1}).\]
It follows that
\begin{equation} \label{Ne1}
 N_e=\frac{\theta(e)}{q^2} \sum_{d\mid e}\frac{\mu(d)}{\phi(d)}\sum_{u,v \in \F}\lambda(-ua-vb)\sum_{\psi_d}S(\hat{\lambda},f_{u,v};\psi_d,g), 
\end{equation}
where $f_{u,v}(x) =ux+vx^{-1}$ and $g(x)=x+x^{-1}$.

Now, evidently, $ S(\lambda,f_{0,0};\psi_1,g) =q^n-1$.  Otherwise, we use Lemma \ref{charest} with $q=q^n$ to bound the sums on the right side of (\ref{Ne1}).  In the most general case, when $d>1$ and $uv\neq 0$, we have $l=3, l_1=l_0=2,  l_2=1$, so that  $|S(\hat{\lambda},f_{u,v};\psi_d,g)| \leq 4q^{\frac{n}{2}}$.  In the other cases, when $uv=0$ or $d=1$ we have 
 $|S(\hat{\lambda},f_{u,v};\psi_d,g)| \leq 2q^{\frac{n}{2}}$.  The
 number of triples $(u,v,d)$, where $u, v\in \F^*$ and $1<d\mid e$,  is
 $(q-1)^2(2^{\omega(e)}-1)$. Further the number of triples with $uv=0$
 and $d\mid e$ is $2^{\omega(e)+1}q$. Hence,  from (\ref{Ne1}) and the fact that there are $\phi(d)$ characters $\psi_d$, 
 \begin{eqnarray*}
N_e &\geq &\frac{\theta(e)}{q^2}\Big{\{}q^n-1- q^\frac{n}{2}\Big[2^{\omega(e)+2}q^{2}-2(q-1)^2-2^{\omega(e)+2}q\Big{]}\Big{\}}\\
 &\geq& \theta(e)q^{\frac{n}{2}} \Big\{  q^{\frac{n}{2}-2} -2^{\omega(e)+2}   \Big{\}}.  
  \end{eqnarray*}
  \end{proof}
  
  We remark that, despite the authors' mistake,  the condition (\ref{Nbound}) for success in  P2 is the same as that given in Theorem 3.3 of \cite{CW} and used by them to obtain their  results.  Coincidentally, therefore,   the results in Table 1 of \cite{CW} remain valid.  As we have noted, however,  their working  is more appropriate to P1 and could have established  the stronger condition $ q^{\frac{n}{2}-2} >2^{\omega(e)+1}$. Indeed, the expression preceding (3.9) in \cite{CW} allows one to make a slight improvement to the bound (3.5).

\subsection{The prime sieve and its application}\label{primesieve}

Let $k$ be a  divisor of $q^n-1$ and $p_1, p_2, \ldots, p_s \  (s \geq 0)$ be the distinct primes dividing $q^n-1$ but not $k$.   We use the following sieving inequality.

\begin{lem}\label{sieve}
Let notation be as above, and
set $\delta=1- \sum_{i=1}^{s} \frac{1}{p_i}$ (with $\delta =1$ if $s=0$).
Then for any $a,b\in\F$,
\begin{equation}\label{sieve1}
N  \geq   \sum_{i=1}^s N_{kp_i} -(s-1)N_k =\sum_{i=1}^s \left[ N_{kp_i}- \left ( 1-\frac{1}{p_i} \right) N_k\right] \ + \  \delta N_k. 
\end{equation}
\end{lem}
From Lemma \ref{sieve} and following the routine transition from a criterion like (\ref{Nbound}) to 
the corresponding sieving condition (see \cite{Co}, \cite{CH} and more
recent work of the second and fourth authors, such as \cite{COST}), we obtain  an extension of the criterion (\ref{Nbound}).
 
 \begin{thm} \label{sieveineq}
 With the assumptions and notation of Lemma $\ref{sieve}$, suppose $\delta>0$ and define $\Delta= \frac{s-1}{\delta} +2$. Then,
 \begin{equation} \label{sieveest}
 N\ge\delta q \theta(q^n-1)q^{\frac{n}{2}}\left\{q^{\frac{n}{2}-2}-4\Delta2^{\omega(k)}\right\}.
 \end{equation}
Hence, whenever
 \begin{equation} \label{sievebound}
 q^{\frac{n}{2}-2}>4\Delta 2^{\w(k)},  \mbox{  that is,  }  q>(4\Delta2^{\w(k)})^{\frac{2}{n-4}},
 \end{equation}  
   then $(q,n) \in Q$.

 \end{thm}
 Before we apply Theorem \ref{sieveineq} to P2, we first establish a small finesse, that will prove particularly useful when we examine the case $n=5$ in \S\ref{sec5}.
 \subsection{The modified prime sieve}
 Continue to let $k$ be a divisor of $q^n-1$, but now let $p_1, \ldots, p_s$ (the {\em sieving primes}) be some of the primes dividing $q^n-1$ but not $k$   and $l_1, \ldots, l_t$ (the {\em large primes}) be the remainder of such primes. Set $ P=p_1\cdots p_s$ and $L=l_1\cdots l_t$.  In useful applications any of the large primes will be significantly greater than any of the sieving primes.  Define $\delta= 1-\sum_{i=1}^s\frac{1}{p_i}$ (as before) and $\varepsilon= \sum_{i=1}^t \frac{1}{l_i}$.
 
 \begin{thm}\label{MPS}
Let notation be as above,
and write the radical of $q^n-1$ as $kPL$.  Suppose also that
$\theta(k)\delta>\varepsilon$ and
 \begin{equation*} \label{MPScrit}
 q^{\frac{n}{2}-2}>\frac{4\{\theta(k)(s-1+2\delta)2^{\omega(k)} +(t-\varepsilon)\}}{\theta(k)\delta- \varepsilon}.
 \end{equation*}
Then $(q,n)\in Q$.
 \end{thm}
 \begin{proof}
 As usual, we can replace $q^n-1$ in $N=N_{q^n-1}$ by its radical, so that $N=N_{kPL}$.   Thus, by Lemma \ref{sieve},
 \begin{equation} \label{cherry} 
  N \geq N_{kP} +N_L-N_1.
  \end{equation}
  Observe that, from (\ref{sieveest}) applied to $N_{kP}$ instead of $N$,
  \begin{equation} \label{NkP}
 N_{kP} \geq \theta(k)q^{\frac{n}{2}}\left\{\delta q^{\frac{n}{2}-2 }-4(s-1+2\delta)\right\}.
  \end{equation}
  By (\ref{sieve1}),
  \begin{equation}\label{grape}
    N_L-N_1 \geq \sum_{i=1}^t\left\{N_{l_i}- \left(1- \frac{1}{l_i}\right)N_1\right\}-\varepsilon N_1. 
    \end{equation}
    
  Now, for the given prescribed traces $a,b \in \F$, $N_1$ enumerates the nonzero $\xi\in \Fn$ with $T(\xi)=a, T(\xi^{-1})=b$,  for which $\xi+\xi^{-1}$ is non-zero.   As in \cite[\S 4]{Gupta},  we have $N_1 \leq q^{n-2}$.
  
  Next, let $l$ be a prime factor of $L$.   Then, from (\ref{Ne1}) and (\ref{grape}),
  
  \[N_l-\left(1-\frac{1}{l}\right)N_1=-\frac{1}{q^2} \sum_{u,v \in \F}\lambda(-ua-vb)\sum_{\psi_l}S(\hat{\lambda},f_{u,v};\psi_l,g).  \]
  Hence, by (\ref{grape}),
  \begin{equation}\label{peach}
  N_L-N_1\geq  - 4(t-\varepsilon)q^{\frac{n}{2}}-\varepsilon q^{n-2}.
  \end{equation}
  The result follows by combining (\ref{cherry}), (\ref{NkP}) and (\ref{peach}).
  \end{proof}

 \section{Problem P2 for $n\geq 5$}\label{6up}
 
Although the derivation of (\ref{Nbound}) as a criterion for the solution of Problem P2 in  \cite{CW} is incorrect, the authors  were able to verify easily a complete solution to Problem P2  for $n \geq 29$.   Further, with the aid of computation 
they extended this to $n \geq 25$ for a complete solution and  to $n \geq 10$ for a partial solution.   In addition to the condition  (\ref{Nbound}),  they employed a general upper bound for $\omega(m)$ based on the prime divisors of $m$ established by Lenstra and Schoof \cite{LS}.  We prefer to use Robin's \cite{Robin} bound for $\omega(m)$ together with  (\ref{Nbound})   and then proceed to use the condition (\ref{sievebound}) for an appropriate choice of sieving primes $p_1, \ldots, p_s$. 
  
\begin{lem}[Robin]\label{W}
  For any $n\geq 3$
  
\begin{equation}   \label{Wbound}
\omega(n) \leq \frac{1.385 \log n}{\log \log n}.
\end{equation}
    \end{lem}
  We remark that for any $\epsilon>0$ one may find an $n_{0}(\epsilon)$ such that (\ref{Wbound}) holds with constant $1+\epsilon$ in place of 1.385, provided that $n\geq n_{0}$. We do not need such variations here, since we merely use Lemma \ref{W} to reduce the infinite problem to a finite number of cases.
  
We set $x = q^{n} -1$ and use Lemma \ref{W} to bound $\omega(x)$ and hence to give the following consequence of (\ref{Nbound}), for $n\geq n_{0}$, namely
\begin{equation}\label{cheese}
(x+1)^{1/2 - 2/n_{0}} > 2^{2 + \frac{1.385 \log x}{\log\log x}}.
\end{equation}
If (\ref{cheese}) is true for $x\geq x_{0}$, say, then, since $q\geq 2$, we have that (\ref{cheese}) is true for $x\geq 2^{n_{0}} -1$. If $2^{n_{0}} -1 > x_{0}$ there is nothing to check, and so we deduce that, for all $n\geq n_{0}$, we have $(q, n) \in Q$ for all $q$. The best choice here is $n_{0} = 25$, which shows that (\ref{cheese}) is true for $x\geq x_{0} \approx 2.5\cdot 10^{7}$, whereas $2^{25} \approx 3.3\cdot 10^{7}$. Hence we instantly prove that for all $n\geq 25$ we have $(q, n) \in Q$ for all prime powers $q$.

Now we consider $n\leq 24$. For each $n\leq 24$ we let $\omega(q^{n} -1) = j$, whence $q_{n} -1 \geq p_{1} \cdots p_{j}$, where $p_{i}$ denotes the $i$th prime. This then gives a lower bound on $q$, namely,
\begin{equation}\label{milk}
q\geq (p_{1} \cdots p_{j} + 1)^{1/n},
\end{equation}
which, combined with (\ref{Nbound}) shows that $(q, n) \in Q$ if
\begin{equation}\label{butter}
(p_{1} \cdots p_{j} + 1)^{1/2 - 2/n} > 2^{2+j}.
\end{equation}
We aim at using (\ref{butter}) for each $5\leq n \leq 24$ to show that $j$ cannot be too large. For example, when $n=24$ we have that (\ref{butter}) is true for all $j\geq 9$, hence, we only need to consider $q$ such that $\omega(q^{n} -1) \leq 8$. At worst, in (\ref{Nbound}) we have $\omega(q^{n} -1) = 8$, and with $n=24$ this means that we need only check $q< 3.53$. This means that the only prime powers at require checking at $q=2,3$. For each of these we check the exact factorisation of $\omega(q^{n} -1)$, and in each case, we find that there is nothing to check. Hence we have eliminated $n=24$. 

We continue in this way and eliminate all $n\geq 17$. We find, though, that we are unable to eliminate $(q, n) = (2, 16)$. For $n\leq 16$ we require the sieve from Theorem \ref{sieveineq}.

Choosing, for example, $s=2$ eliminates the need for further examination of $(q, n) = 16$, whence we need only consider those $n\leq 15$. We continue like this for all $n\geq 6$ and produce the following table of possible exceptions.

\begin{table}[tbh]
\centering
\caption{Possible exceptions to $(q, n) \in Q$.}\label{Tab1}
\begin{tabular}{|cc|}\hline
$n$ & List of prime powers $q$\\\hline
 12 & $2$  \\
    11 & $2$ \\
      10 & $2$  \\
      9 & $2,3,4$  \\
        8 & $2, 3, 4, 5, 7, 8$  \\
          7 & $2, 3, 4, 5, 7, 11$ \\
            6 & $S_{6}$  \\\hline
\end{tabular}
\end{table}

We note that $S_{6}$ is the following set of 49 elements: 
\begin{equation}\label{cream}
\begin{split}
S_{6} =  &\{2, 3, 4, 5, 7, 8, 9, 11, 13, 16, 17, 19, 23, 25, 27, 29, 31, 32, 37, 
41, 43, 47, 49, 53, 59, 61,64, 67,\\  & 71, 73, 79, 81, 83, 89, 101, 103, 
107, 109, 113, 121, 131, 137, 139, 149, 169, 179, 181,191, 211\}.
\end{split}
\end{equation}

\section{The case $n=5$ and the hybrid bound}\label{sec5}
The case $n=5$ provides some additional difficulties. We can, through judicious choice of $s$, win for all $\omega(q^{5} -1) \geq 24$. When $\omega(q^{5} -1)= 23$, the best choice is $s=3$, which shows that we need to consider prime powers $q< 6.2\cdot 10^{6}$. While it is possible to enumerate these prime powers, and then to check them against (\ref{sievebound}), we provide an alternative approach, which helps with the problem considered \S \ref{misc}, and which may be of use for similar problems.

Throughout the use of the sieve we have utilised a lower bound on $q$ in
(\ref{milk}). We strengthen this with the following observation. Note
that if $ p\mid  q^{5} -1$ and $ p  \not\equiv 1 \pmod {10}$, then
$p\mid  q-1$. Define $r_{i}$ as the $i$th prime congruent to $1\pmod{10}$ and $s_{i}$ as the $i$th prime not congruent to $1\pmod {10}$. For $\omega(q^{5} -1) = j$ suppose there are exactly $m$ primes $\equiv 1\pmod{10}$ dividing $q^{5}-1$, with $0\leq m\leq j$. We therefore have that each of these $m$ primes divide $q-1$, whence 
\begin{equation}\label{custard1}
q-1 \geq r_{1} \cdots r_{m},
\end{equation}
and also
\begin{equation}\label{custard2}
q^{5} -1  \geq \left(r_{1} \cdots r_{m}\right) \left( s_{1} \cdots s_{j-m}\right).
\end{equation}
Rearranging (\ref{custard2}) gives a lower bound on $q$, which, when combined with (\ref{custard1}) gives us a hybrid lower bound for $q$. Taking the minimum over all possible values of $m$ gives us the following
\begin{equation}\label{curd}
q\geq \min_{0\leq m\leq j} \max\left\{ r_{1} \cdots r_{m} +1, \left(r_{1} \cdots r_{m} s_{1} \cdots s_{j-m} +1\right)^{1/5}\right\}.
\end{equation}
To show the utility of this bound, consider $\omega(q^{5} - 1) = j = 23$. From (\ref{milk}) we have that $q> 3.0\cdot 10^{6}$, but, from (\ref{curd}) we have $q> 3.8\cdot 10^{7}$, coming from the worst case of $m=8$. The hybrid lower bound in (\ref{curd}) allows us to eliminate $j=23, 22, 21$. Therefore we need only consider $j\leq 20$ which, upon choosing $s=17$, means looking only at those $q< 3.4\cdot 10^{6}$.

We now consider $q$ even: since $q< 3.4\cdot 10^{6}$
we have that $q=2^{\alpha}$ with $\alpha\leq 21$. For each of these values we can use exact factorisation, which shows that the only even $q$ that require further testing are 
\begin{equation*}\label{sugar}
\{2, 4, 8, 16, 32, 64, 128, 256, 512, 1024, 4096\}.
\end{equation*}
These can be tested individually using the exact prime factorisation of
$\omega(q^{5} -1)$. Henceforth we shall only be concerned with $q$ odd.

\subsection{Divisor tree}\label{sec:branch}
We cut down on the number of possible exceptions further by checking
only a thin subset of prime powers. This is a similar idea to that
explored by several authors \cite{MTT}, \cite{Jarso}, \cite{Jarso2}.
Note that, if small primes do not divide $q^{5}-1$, not only does the
lower bound for $q^{5} -1$ increase via (\ref{curd}), but the value of
$\delta$ increases, which makes our sieve criterion (\ref{sievebound})
easier to solve. For example, if 3 does not divide $q^{5} -1$ then we
have $q> 5.3\cdot 10^{6}$ whereas the sieve criterion (\ref{sievebound})
is satisfied for $q> 4.9\cdot 10^{5}$. Therefore, there is nothing to
check in this case, whence we deduce that $3\mid q^{5} -1$. 

We continue in this way, showing that also $5, 7$ and $13$ divide $q^{5} -1$ --- the prime $2$ divides as well, of course, since $q$ is odd. We may therefore write 
\begin{equation*}\label{yoghurt}
q = 2730M +1, \quad q< 3.4\cdot 10^{6} \Longrightarrow M\leq 1245.
\end{equation*}
We now enumerate these 1245 values of $q$, retaining only those $q$ that are prime powers, as well as those $q$ for which $\omega(q^{5} -1) = 20$. We find no such $q$, and hence, we can eliminate $j=20$. We continue in this way, eliminating all $15\leq j \leq 19$.

When $j=14$ we find exactly one prime power $q = 652,831$ which meets the requirement that $\omega(q^{5} -1) = 14$. However, we can use exact factorisation and take any $10\leq s \leq 13$ to eliminate this value of $q$. We similarly eliminate $j=13$.

When $j=12$ we find our first `genuine survivor': the entire arsenal of theoretical machinery fails to show that $q= 27691$ qualifies for membership of $Q$. We continue in this way and arrive at 816 possible values of $q$ to check, the largest of which is $q=63211$. 
Rather than give this list, we proceed to use the modified prime sieve, Theorem \ref{MPS}, as in \cite{Bailey}, \cite{Booker}, to remove some of these exceptions. Doing this removes 52 elements from our list of 816 possible exceptions when $n=5$.

We note that we can use Theorem \ref{MPS} to remove the case $(n, q) =
(8, 7)$, by taking $s=1$, $r=1$. Similarly, we can remove $q=113$, $169$,
and $191$ from our list of possible exceptions for $n=6$ in (\ref{cream}).


\section{Computation}\label{comp}
We provide details of the computation used to show the remaining pairs $(q, n) \in Q$. Recall that we need to cover all possible pairs $(a, b)$ in the ground field $\mathbb{F}_{q}$ and show we can find $\xi$ such that $\xi + \xi^{-1}$ is primitive with $T(\xi) = a$ and $T(\xi^{-1}) = b$. 

Our basic procedure is described in pseudocode in Algorithm~\ref{Ag1}.
We note a few features:
\begin{itemize}
\item
As written the algorithm cannot detect failure---if no suitable $\xi$
exists then it will enter an infinite loop. To circumvent this, one
could count the number of random trials tried for a given pair $(a,b)$,
and perform an exhaustive search over all choices of the coefficients
$c_k$ when the count exceeds $q^{n-2}$, say. In any case, we only
encountered one true exception, $(q,n)=(2,6)$, which is easily seen to
be exceptional for the pairs $(a,b)=(0,0)$ and $(1,1)$ by exhausting
over all $\xi\in\mathbb{F}_{2^6}$.
\item
The probability that the polynomial $P$ produced in step~\ref{randstep}
is irreducible is about $1/n$, and the probability that $\gamma$ defined
in step~\ref{gamma} is primitive is then about $\varphi(q^n-1)/(q^n-1)$,
so we expect to succeed in about $n(q^n-1)/\varphi(q^n-1)$ trials
on average. This is largest for $(q,n)=(121,6)$, where it equals
$31.39\ldots$.
\item
We can stop the inner loop at $j=i$ thanks to the symmetry
between $a$ and $b$ (swapping them is equivalent to swapping $\xi$ and
$\xi^{-1}$). A further reduction is possible when $q\equiv1\pmod4$,
since negating a solution for $(a,b)$ yields one for $(-a,-b)$, but we did
not make use of this in our implementation.
\end{itemize}

\begin{algorithm}[htbp]
\label{Ag1}
\SetAlgoNoLine
\DontPrintSemicolon
Input: $q$ and $n$.

Output: success if $(q,n) \in Q$; no output otherwise.

Construct the finite field $\F$ and generate a random primitive element,
denoted $g$.

Generate a list of all elements $\alpha_i\in\F$, where
$\alpha_0=0$ and $\alpha_i=g^i$ for $1\le i\le q-1$.

\For{\upshape $i$ from $0$ to $q-1$}{
	\For{\upshape $j$ from $0$ to $i$}{
		Set $a=\alpha_i$, $b=\alpha_j$.

		Choose random $c_0\in\F^*$, $c_2,\ldots,c_{n-2}\in\F$
		and set $P=x^n-ax^{n-1}+\sum_{k=2}^{n-2}c_kx^k-c_0bx+c_0$.
		\label{randstep}

		\lIf{\upshape $P$ is reducible}{go to step \ref{randstep}.}

		Set $\xi=x+(P)\in\F[x]/(P)$ and
		$\gamma=\xi+\xi^{-1}=b-c_0^{-1}\bigl((c_2-c_0)x+\sum_{i=2}^{n-3}c_{k+1}x^k
		-ax^{n-2}+x^{n-1}\bigr) + (P)$.
		\label{gamma}

		\For{\upshape each prime $p\mid q^n-1$}{
			\lIf{$\gamma^{(q^n-1)/p}=1$}{go to step \ref{randstep}.}
		}
  }
}

\Return{success.}

\caption{Verify that $(q,n) \in Q$}
\end{algorithm}

\subsection{Potential exceptions for direct verification}
For $n=5$, there are $764$ values of $q$ that our sieves fail to
eliminate, including all $q\le1181$; the largest is the prime
$62791$. See Table~\ref{Tab1} for candidate exceptions when $6\le n \le 12$.

\subsection{Computational results}\label{wallet}
We coded two versions of Algorithm~\ref{Ag1}. The first, in
\texttt{PARI/GP} \cite{PARI}, was used to handle all pairs with $n\ge6$,
as well as those with $n=5$ and $q\le7$. The second, faster version
was coded in \texttt{C} using the \texttt{FLINT} library \cite{FLINT}
for arithmetic in $\F[x]$. Specifically we used \texttt{FLINT}'s routines
based on Zech logarithms, which are very fast when $q$ is small.  (In this
context even our largest value $q=62791$ counts as ``small''.) For this
version we restricted to $n=5$ and $c_0=-1$ in order to simplify the
coding, so the $\xi$ that it produces always have norm $1$; although it
is not necessary to impose this restriction, doing so causes no harm,
as it seems to be sufficient for all $q\ge8$.

The total running time to check all candidates was approximately 6 days
using 192 cores (AMD Opteron processors running at 2.5~GHz).

\section{Another application of the hybrid bound}\label{misc}

Let us say that $q\in \mathfrak{P}$ if, for any $a\in \mathbb{F}_{q}$ we
can find a primitive element $\xi\in \mathbb{F}_{q^{3}}$ such that $\xi +
\xi^{-1}$ is also primitive and  $T(\xi) = a$. The analogous
problem with $\xi\in \mathbb{F}_{q^{n}}$ was considered by Gupta, Sharma,
and Cohen \cite{Sharma} who proved a complete result for $n\geq 5$. Cohen
and Gupta \cite{Gupta} extended this work, providing a complete result for
$n=4$, and, for $n=3$, showing that $q\in \mathfrak{P}$ if $\omega(q^{3}
-1) \geq 27$ or $q\ge 3\cdot10^{13}$. They did this with the following
--- see \cite[Thm.\ 3.2]{Gupta}.

\begin{thm}[Cohen and Gupta]\label{oval}
Let $q$ be a prime power and write $q^{3} -1 = kP$, where $(k, P) = 1$ and $\omega(P) = s\geq 1$. Define
\begin{equation*}\label{trumpet}
\delta  = 1 - 2\sum_{i=1}^{s} \frac{1}{p_{i}},
\end{equation*}
where $\rad(P) = p_{1} \cdots p_{s}$ for distinct primes $p_{i}$. Let $C_{q} = 2$ if $q$ is even and $C_{q} = 3$ if $q$ is odd. If $\delta>0$ and
\begin{equation}\label{trombone}
q^{1/2} > C_{q} 2^{2\omega(k)} \left( 2 + \frac{2s-1}{\delta}\right),
\end{equation}
then $q\in \mathfrak{P}$.
\end{thm}
While Cohen and Gupta consider also a modified prime sieve, analogous to our Theorem \ref{MPS}, we do not include this here. Instead, we proceed to use (\ref{trombone}) for $\omega(q^{3} -1) \leq 26$.

Proceeding directly, we first dispense with $q$ even,
for which we need only consider $q=2^{\alpha}$ for $\alpha
\leq\lfloor\log_2(3\cdot10^{13})\rfloor=44$.  When we factor $q^{3} -1$
explicitly and apply the sieve, the only values of $q$ that survive and
need further testing are
\begin{equation*}\label{dollop}
q = 2^{\alpha} \quad\text{for }\alpha \in \{1 ,2, 3, 4, 6, 8, 10, 12, 20\}.
\end{equation*}

We now focus on $q$ odd. We improve on the lower bound of $q>3\cdot
10^{13}$ with a version of the hybrid lower bound (\ref{curd}). We make
two observations that help with the lower bound. First, that if
$p\mid q^{3}-1$, where $p=3$ or $p\equiv 2\pmod3$ then $p\mid q-1$. Second, that
if we know that $3\mid q^{3} -1$ then we must also have that
$3^{2}\mid q^{3}-1$. We use the first of these to establish an analogy with (\ref{curd}),
namely that if $\omega(q^{3}-1) = j$, then
\begin{equation}\label{whey}
q\geq \min_{0\leq m\leq j} \max\left\{ t_{1} \cdots t_{m} +1, \left(t_{1} \cdots t_{m} u_{1} \cdots u_{j-m} +1\right)^{1/3}\right\},
\end{equation}
where $t_i$ denotes the $i$th prime congruent to $0$ or $2\pmod3$
and $u_i$ denotes the $i$th prime congruent to $1\pmod3$. Similarly,
if $3\mid q^{3} -1$ then the second term in (\ref{whey}) is multiplied
by $3^{1/3}$.

This eliminates $\omega(q^{3} -1) = 26$ or $25$ without any
trouble. When $\omega(q^{3} -1) = 24$ we find that we need to check those
$q\in (2.87\cdot 10^{11}, 1.77\cdot 10^{13})$. With the hybrid bound
this translates to checking only $m=7$, $8$, $9$, $10$. We easily dispense with
$m=7$, $8$. With $m=9$, $10$ we need to examine whether some small primes do,
or do not divide $q^{3}-1$. While this works for $\omega(q^{3} -1) = 24$,
the number of cases to check explodes when we examine $\omega(q^{3}-1)
= 23$. Hence we conclude that we win for all $\omega(q^{3} -1)\geq
24$. When $\omega(q^{3} -1) \leq 23$ we win provided that $q\geq 8\cdot
10^{12}$. This gives the following result, an improvement on Theorem
9.1 of  \cite{Gupta}.

\begin{thm}
We have that $q\in \mathfrak{P}$ if $\omega(q^{3} -1) \geq 24$ or if $q\geq 8\cdot 10^{12}$.
\end{thm}

Further considerations  involving, for example, the distribution of primes that are congruent to 1 modulo 6 as factors of $q-1$ and $q^2+q+1$ can be undertaken. Inevitably, however, the complete resolution of this question will involve substantial computation that would be inappropriate in the present article.

\subsection*{Acknowledgements}
We are grateful to David Harvey for discussions on the problem in \S \ref{misc}.

\end{document}